\documentclass[11pt,letterpaper]{amsart}
\usepackage{amsmath}
\usepackage{amssymb}
\usepackage{dsfont}
\usepackage{hyperref}
\usepackage{xcolor}
\usepackage{graphicx}

\newtheorem{theorem}{Theorem}[section]
\newtheorem{corollary}[theorem]{Corollary}
\newtheorem{lemma}[theorem]{Lemma}
\newtheorem{proposition}[theorem]{Proposition}
\newtheorem{conjecture}[theorem]{Conjecture}
\theoremstyle{definition}
\newtheorem{definition}[theorem]{Definition}
\newtheorem{example}[theorem]{Example}
\numberwithin{equation}{section}
\newcommand{\R}{\mathbb{R}}
\newcommand{\Z}{\mathbb{Z}}
\newcommand{\A}{\mathbf{A}}
\newcommand{\D}{\mathcal{D}} 
\newcommand{\E}{\mathcal{E}}
\newcommand{\J}{\mathcal{J}}
\newcommand{\W}{\mathcal{W}}
\newcommand{\one}{\mathds{1}}

\begin{document}

\title{Translational self-affine tiles in $\R^n$}

\author{Li-Xiang An}
\address{School of Mathematics and Statistics and Hubei Key Laboratory of Mathematical Sciences, Central China Normal University, Wuhan 430079, P.R. China}

\email{lixiangan@ccnu.edu.cn}

\author{Ting-Ting He$^{\ast}$}

\address{School of Mathematics and Statistics and Hubei Key Laboratory of Mathematical Sciences, Central China Normal University, Wuhan 430079, P.R. China}

\email{hetingting0816@163.com}

\thanks{$\ast$Corresponding author.}
\subjclass[2020]{Primary 11B75, 52C22; Secondary 11A63, 28A80}
\keywords{self-affine tiles, tile digit sets, product-form digit sets, periodic tilings}

\begin{abstract}
Two basic problems in the theory of integral self-affine tiles are to characterize the digit sets that generate translational tiles and to understand whether such tiles admit periodic translational tilings. We give a complete solution to the first problem: for an expanding integer matrix $\A$ and a digit set $\D\subset\Z^n$ with $\#\D=|\det\A|$, the self-affine set $T(\A,\D)$ is a tile if and only if, for some $N\geq1$, the iterated digit set $\D_{{\A},N}$ is of skew-product form with respect to $\A^N$. As a consequence, every integral self-affine tile admits a fully periodic self-replicating $q-$fold tiling. Moreover, $q$ divides $|T(\A,\D)|$. In dimension two we prove more: every integral self-affine tile admits a fully periodic tiling. In particular, the periodic tiling conjecture holds for integral self-affine tiles in $\R^2$.
\end{abstract}

\maketitle

\section{Introduction}

\subsection{Self-affine tile} Let $\A\in M_{n}(\Z)$ be an expanding matrix (i.e., all eigenvalues have moduli $> 1$), and let $\D = \{{d}_0,  {d}_1, \cdots {d}_{m-1}\} \subset {\Z}^n$ be a finite set with $\#\D=|\det \A|$.  The {\it affine pair} $(\A, \D )$ defines an iterated function system $\{\varphi_i\}_{i=0}^{m-1}$ with
$\varphi_i (x) = \A^{-1} (x+ {d}_i)$.
Results of Hutchinson \cite{Hut81} then imply that there is a unique nonempty compact set $T := T(\A,\D )\subset {\R}^n$ ({\it self-affine set}) satisfying the set-valued functional equation
$$
\A T = T+\D .
$$
It is well-known that when $T$ has non-empty interior, then $T$ is a translational tile \cite{B91}, that is, there exists a discrete set $\J \subset \R^n $ such that
$$
\sum_{j\in \J}\one_{T}(x-j) =1 \quad \text{ for almost every } x\in  \R ^n,
$$
where $\one_{E}(x)$ denotes the  indicator function of $E$. We call such $T$  an {\it integral self-affine tile}, $\D $  a {\it tile digit set} with respect to $\A$ and $\J$ a {\it tiling set} of $T$. In this case, $\#\D=|\det \A|$ is necessary.

The self-affine tile has its origin in the study of finite automata and recurrent sets \cite{K92-2, Rau82,Thu89}. The development was strongly motivated by wavelet theory, fractal geometry and geometry of numbers \cite{ABBPT05, AB06, B91, BS11, CT16, GY06, HL04, KL1, KL2, LW96-1, LW96-2, LW96-3, LW97, LL, SW}, and the connection with the Fuglede's conjecture on tiles and spectral sets \cite{FHL15, F74, Ko, KoM, T}. Despite the large literature on self-affine tiles, there is a fundamental question that is still widely open:

\medskip
\noindent

\textbf{(Question).} For a given  expanding matrix $\A\in M_{n}(\Z)$, characterize the finite sets $\D \subset \Z^n$ for which  $T(\A, \D)$ is a self-affine tile. 

\medskip

As is known, the most basic tile digit set is a complete residue set modulo $\A\Z^n$ \cite{B91}. Kenyon \cite{K92-2} proved that on ${\R}$,  the converse is also true if $\A  = [p]$ is a prime; it was extended to ${\R}^n$ with a mild assumption that $\D $ spans ${\R}^n$ \cite{HL04} (see also \cite{LW96-1} for another more restricted condition). But complete residue sets are far from the whole story, even in one dimension. A typical example is

\begin{example}
$\D =\{0,1,8,25\}$ is a tile digit set for $\A=[4]$. Although $\D$ is not a complete residue system modulo $4 \Z$, it decomposes as
$$
\D 
=
\bigl(\{0\}+4\{0,2\}\bigr)
\cup
\bigl(\{1\}+4\{0,6\}\bigr),
$$
and
$$
\{0,1\}\oplus\{0,2\}
\equiv
\{0,1\}\oplus\{0,6\}
\equiv
\{0,1,2,3\}\pmod{4 \Z}.
$$
Thus the missing residue classes at the first level are recovered through
different fibers at the next scale.   
\end{example}
This phenomenon led to the introduction of {\it product-form digit sets} by Lagarias and Wang \cite{LW96-1}: 
$$
\D = \E_0 \oplus \A^{\ell_1}\E_1 \oplus \cdots \oplus \A^{\ell_k} \E_k,\quad (\hbox {mod} \ \textbf{A}^{\ell_k+1}\Z^n),
$$
where
$\E = \E_0  \oplus \E_1 \cdots \oplus \E_k$ is a complete residue set modulo $\A\Z^n$.  In \cite{LLR13,LLR17,LR03}, Lai, Lau and Rao further extended the above form to certain {\it modulo product-forms} and {\it higher order modulo product-forms} in $\R$. More recently, Lau and the first-named author \cite{AL19} discovered that the factor appearing at a later scale is allowed to depend on the digit chosen at an earlier scale. Li and Rao \cite{LR25} formulated this dependence in terms of {\it skew-product-form digit sets} (see Definition~\ref{def-spf}) and proved that in $\R$, every tile digit set acquires such a structure after finitely many iterations. A crucial ingredient in their argument is the existence of a tiling set that is both
periodic and self-replicating. However, such periodicity is no longer available in the required generality, and little is known about the structure of tile digit sets in higher dimensions. In this paper, we will give a complete answer to \textbf{(Question)} in arbitrary dimensions. As a consequence,  we will use the structure of the digit set to derive  periodicity and self-replication properties of tiling sets.

\subsection{Tiling set and periodicity}  We call $\J\subset \R^n$ {\it self-replicating} with respect to $(\A, \D)$ if 
$$\J=\A\J+\D,$$ 
and  {\it fully periodic} if it is invariant under translations by some full-rank  discrete subgroup of $\R^n$.  A major conjecture in this area is the periodic tiling conjecture (PTC).

\begin{conjecture}[PTC]
Every bounded measurable translational tile in $\R^n$ must admit a fully periodic tiling set.   
\end{conjecture}

This conjecture is true in dimension one, by the work of Lagarias and Wang \cite{LW96-3}. In sufficiently high dimensions, it was recently disproved by Greenfeld and Tao \cite{GT24}. They constructed a finite set $B\subset\Z^n$ which tiles $\Z^n$ but admits no fully periodic tiling set and this yields an aperiodic translational tile $R_n+B$ in $\R^n$, where $R_n$ is a ``rigid" fundamental domain for $\Z^n$. A translational tile is called \emph{aperiodic} if it admits a translational tiling but admits no fully periodic tiling. Nevertheless, it remains valid in several important structured and low-dimensional settings. For instance, convex polytope in $\R^n$ \cite{McM80, Ven54},  topological disks \cite{BN91,GBN91} and rational polygonal sets in $\R^2$ \cite{DGGM26}.

Self-affine tiles possess an additional dilation structure. Indeed, the functional equation 
$$
\A T=T+\D
$$
iterates to
$$
\A ^kT=T+\D_{\A ,k},
\qquad
\D_{\A ,k}
:=
\D+\A \D+\cdots+\A ^{k-1}\D.
$$
The translational tiling of $T(\A,\D)$ is not arbitrary, but is closely tied to the arithmetic structure of the finite digit expansions. This additional structure is already visible for polyhedral self-affine tiles. The classification of Protasov and Zaitseva \cite{PZ23} implies that self-affine tiles of the form $R_n+B$ must admit a fully periodic translational tiling set.  In one dimension, Kenyon \cite{K92-2}, Lau and Rao \cite{LR03} showed moreover  that self-similar tiles always produce unique periodic and self-replicating tiling $\J$ in $\Z$. That is,  for a tile $T(b, \D)$ with $0\in\D\subset\Z$ and $\gcd(\D)=1$, if $S\subset\Z$ is periodic and self-replicating, then $S=\J$ and it is a tiling  of $T(b, \D)$. This result is convenient to use to construct a periodic tiling. Indeed, if $\D_{b,N}$ is of skew-product form for some $N\geq1$, then the periodic set
$\J=\D_{b,N}+b^N\Z$ is self-replicating. It follows that $T(b,\D)$ tiles $\R$ by translations in $\J$. This is  the result of  Li and Rao \cite{LR25}.
These results provide further evidence that the finite-level arithmetic structure imposed by self-affinity plays a fundamental role in enforcing periodicity of the associated tilings.
In \cite[Conjecture 1.1]{LW96-1}, Lagarias and Wang proposed the following fundamental question.

\begin{conjecture}
Every integral self-affine tile must admit a fully periodic tiling set contained in  $\Z^n$.
\end{conjecture}

However, the situation in higher dimensions is substantially more subtle. It is still  not known in general whether integral self-affine tiles always produce  a set which is both fully periodic and self-replicating.   Lagarias and Wang \cite{LW96-2} already pointed out that these two desirable properties ``fully periodic'' and ``self-replicating'' cannot in general be imposed simultaneously on a tiling set.  The following classical example illustrates the obstruction.

\begin{example}\cite[Example~2.3]{LW96-2}\label{example}
Let $\A=\scalebox{0.7}{$\begin{pmatrix}
2&1\\
0&2
\end{pmatrix}$}$
and
$\D  =\{0,3\}\times\{0,1\}.$
Then $\D $ is a complete residue system modulo
$\A \Z^2$ and $T({\bf A}, \D)$ admits the lattice tiling
$3\Z\times\Z.$ If we assume that $\J\subset\Z^2$ which is fully periodic and satisfies $\J={\bf A}\J+\D$, then $\J=\Z^2$ since $\D $ is a complete residue system modulo
$\A\Z^2$. However, $\Z^2$ is not a tiling set of $T$; instead, it gives a $3$-fold tiling. Consequently, this integral self-affine tile admits no tiling set which is simultaneously fully periodic and self-replicating.
\end{example}

Here a set $\J \subset \R^n $ is called a $q$-fold tiling set of a measurable set $T$ if 
$$
\sum_{j\in \J}\one_{T}(x-j) =q \quad \text{ for almost every } x\in  \R ^n.
$$
Thus, although periodicity and self-replication may be incompatible at the one-fold tilings, they may coexist after allowing a multiplicity tiling. 

Our second main result is that, in arbitrary dimension, every integral self-affine tile  admit a fully periodic and self-replicating $q$-fold tiling set for some integer $q\geq1$. In dimension two, using a recent theorem of Greenfeld and Tao \cite{GT25} on the periodic multi-tiling in $\Z^2$, we can further obtain a fully periodic tiling set. In particular, the periodic tiling conjecture holds for integral self-affine tiles in $\R^2$.

\subsection{Main results} 
Before stating our main results, we introduce some notation.  For $x\in\R,$ write $e(x)=e^{2\pi i x}$. For $x=(x_1, \cdots, x_n)\in \R^n$, we use the same notation 
$$e(x)=(e(x_1), \cdots, e(x_n))\in (\mathbb C^\times)^n.$$  
For a finite set $E\subset\Z^n$, define its Laurent polynomial by
$$
P_E(z)=\sum_{d\in E}z^d,\quad z^d=z_1^{d_1}\cdots z_n^{d_n}.
$$
For $k\geq0$, set
$$
e_k(x):=e\bigl((\A^T)^{-k}x\bigr),
$$
where $e_0(x)=e(x)$. Finally, for $k\geq1$ define
$$
\mathcal Z_k(\A,\D )=\left\{v\in\Z^n: P_\D \left(e_k(v)\right)=0\right\}.
$$
Kenyon's criterion provides a Fourier characterization of tiles. In its primitive form it states  (see Corollary \ref{coro-Kenyon}) that
\begin{equation*}
T(\A,\D )  \text{ is a tile}
\quad\Longleftrightarrow\quad
\Z^n\setminus \A^T\Z^n
\subset
\bigcup_{k=1}^{\infty}\mathcal Z_k(\A,\D ).
\end{equation*}
We sharpen this criterion to a finite version, which is the central analytical tool for characterizing tile digit sets.

\begin{theorem}[Finite zero-covering theorem]\label{fzp-0}
Let $\A\in M_n(\Z)$ be expanding, $\D \subset\Z^n$ with
$\#\D =|\det \A|$. The self-affine set $T(\A,\D )$ is a tile if and only if there is an $N\geq1$ such that 
\begin{equation}\label{eq-Z}
\Z^n\setminus \A^T\Z^n
\subset
\bigcup_{k=1}^{N}\mathcal Z_k(\A,\D ).
\end{equation}
\end{theorem}

We next formulate the structural consequence. 
If a family of sets is pairwise disjoint, we use \(\bigsqcup\) to denote their union.

\begin{definition}[Skew-product-form digit set {\cite{LR25}}]\label{def-spf}
Let $\A$ be an $n\times n$ expanding integer matrix, and let $\D \subset\Z^n$ satisfy $\#\D =|\det \A|. $
We say that $\D $ is a {\it skew-product-form digit set} with respect to $\A$ if there exist $\mathcal W=\{w_j\}_{j=0}^{s-1}\subset\Z^n $ and finite sets $\mathcal E_0,\ldots,\mathcal E_{s-1}\subset\Z^n$
such that
\begin{equation*}
\D =\bigsqcup_{j=0}^{s-1}\left(
\{w_j\}+\A\mathcal E_j
\right),
\end{equation*}
and
$\mathcal W\oplus\mathcal E_j$ is a complete residue set modulo $\A\Z^n$ for every $j=0,\ldots,s-1$. 
\end{definition}

When each $\E_j$ consists of a single element, then $\D$ itself is a complete residue set modulo $\A\Z^n$. More examples and discussions may be found in \cite{AL19} \cite{LR25}. Recall that the iterated digit  associated to $(\A, \D)$:
$$
\D_{\A, k}=\D +\A\D +\cdots+\A^{k-1}\D .
$$
Then $$T(\A, \D)=T(\A^k, \D_{\A, k}),$$ 
so that passing to a finite iteration does not change the underlying tile. Our first main theorem shows that the tile digit set is of skew-product form in finite level iterations.

\begin{theorem}\label{main2}  Let $\A\in M_n(\Z)$ be expanding, $\D \subset\Z^n$ with
$\#\D =|\det \A|$. Then $T(\A,\D )$ is a self-affine tile if and only if  $\D _{\A,N}$ is a skew-product-form digit set with respect to $\A^N$ for an $N\geq1$ satisfies Theorem \ref{fzp-0}. 
\end{theorem}

It is worth emphasizing that the finite-level structure is stable once it appears. Indeed, if \(N_0\) satisfies the finite zero-covering condition (1.1), then so does $N$ with $N\geq N_0.$ Consequently, \(\D_{\A,N}\) is of skew-product form with respect to $\A^N$ for every $N\geq N_0$.

For the one-dimensional case, the lower bound  $N_0$ can be described explicitly in terms of the diameter of  $\D$. More precisely, for any $x\in\R$, we use $\lfloor x\rfloor$ to denote the greatest integer less than or equal to $x$. One may take 
$$N_0= 1+{\max}_{d\in(\D-\D)\setminus\{0\}}\left\lfloor \log_2 d\right\rfloor.$$ 
Actually, the lower bound \(N_0\) is optimal. For instance, $\D=\{0, 1\}$ a complete residue set modulo $2$, in which  $N_0=1$.
Since the proof of this refinement requires a separate argument based on cyclotomic polynomial theory, we postpone it, together with several related remarks, to a later section.

The finite-level
skew-product structure also produces an $\A ^N\Z^n$-periodic and self-replicating set $\J_R\subset\Z^n$, that is  
$$\J_R=\W+{\bf A}^N\Z^n=\D+\A\J_R.$$
Combining this equivalence with the integer-fiber decomposition yields the following result.

\begin{theorem}\label{multi-tiling}
If $T(\A,\D)$ is a self-affine tile, then it admits a fully periodic self-replicating $q$-fold tiling set $\J_R$ in $\Z^n$. Moreover, $q$ divides $|T(\A,\D)|$.
\end{theorem}

Thus, an integral self-affine tile always produces a multiple tiling set which is both fully periodic and self-replicating. In $\R^2$, the multiple tiling can be replaced by a $1$-fold fully periodic tiling. The key additional input is the periodicity theorem of Greenfeld and Tao \cite{GT25} for indicator-function solutions
of convolution equations on $\Z^2$. Applying that theorem simultaneously to the finitely many fiber
tilings gives the following.

\begin{theorem}\label{thm:planar-ptc}
If $T(\A,\D)$ is a self-affine tile in $\R^2$,  then it admits  a fully periodic tiling set $\J_P\subset\Z^2$. 
\end{theorem}

The last two theorems emphasize a distinction that is important in higher-dimensional self-affine tiling theory. Theorem \ref{multi-tiling} preserves the self-replication identity in every dimension, but may produce a multiple tiling. Theorem \ref{thm:planar-ptc} gives a $1$-fold fully periodic tiling in the plane, but the resulting tiling set need not be self-replicating.

The rest of the paper is organized as follows. Section~\ref{sec:preliminaries} collects the required facts from Fourier analysis on finite abelian groups, Kenyon's criterion and the Noetherian geometry of the algebraic torus. Section~\ref{sec:finite-zero-covering} proves the finite zero-covering theorem. Section~\ref{sec:skew-product} establishes the finite-level skew-product characterization and the existence of fully periodic self-replicating multiple tilings. Section~\ref{sec:planar} proves the periodic tiling theorem in dimension two. Finally, Section~\ref{sec:examples} presents remarks  and  examples. 

\subsection{Disclosure of AI use}

At an early stage of this work, we had already proved that, assuming the finite zero-covering  condition, the tile digit set acquires a skew-product form after finitely many iterations  (Theorem \ref{main2}), and that this further gives rise to a fully periodic and self-replicating multiple tiling set (Theorem \ref{multi-tiling}). In the one-dimensional case, we were able to establish the finite zero-covering  condition by using the theory of cyclotomic polynomials (Theorem \ref{cyc-poly}). This naturally led us to ask ChatGPT-5.6 Sol whether there is a higher-dimensional analogue of the role played by cyclotomic polynomials in this argument. It pointed us to the Zariski topology and the Noetherian property of the algebraic torus (see subsection \ref{section:Noetherian}). With this encouraging outcome, we continued to use AI to explore possible refinements. After further exploration, this led us to prove Theorem \ref{fzp-0}. The final proof was written by us and we take full responsibility for all mathematical claims, proofs, and the final content of the paper.

\subsection{Acknowledgment} The authors are grateful to Professor Chun-Kit Lai from San Francisco State University for bringing the periodic tiling conjecture (PTC) to our attention, and Professor Xing-Gang He from Central China Normal University for many helpful discussions during the preparation of this manuscript. The work was supported by the NSFC grant 12171181.

\section{Preliminaries}\label{sec:preliminaries}

In this section, we collect several basic facts that will be used in
the sequel. Throughout the paper, let $\A\in M_n(\Z)$ be an expanding integer matrix and $\D\subset\Z^n$ with $\#\D=|\det \A|=:m$.

\subsection{Fourier analysis on finite abelian groups} Let $ \Z^n_{\A} :=\Z^n/\A\Z^n$. Then $\# \Z^n_{\A} =m$, and its dual group is naturally identified
with
$$
\widehat{ \Z^n_{\A} }
\simeq
\Z^n/\A^T\Z^n.
$$  More precisely, for $v\in\Z^n$, the corresponding character is
$$\chi_{v}([x])=e(x\cdot(\A^T)^{-1}v),\quad\quad [x]\in  \Z^n_{\A} .$$
Two vectors $v,v'\in\Z^n$ determine the same character if and
only if
$$
v\equiv v'\pmod{\A^T\Z^n},
$$
and $\chi_v$ is nontrivial precisely when $v\notin \A^T\Z^n.$
For every finite set $E\subset\Z^n$,
$$
\sum_{x\in E}\chi_v([x])
=
\sum_{x\in E}e_1(v)^x
=
P_E(e_1(v)).
$$
The following elementary consequence of Fourier analysis on $G_{\A}$
will be used repeatedly.

\begin{definition}
  Let $X\subset\Z^n$ be a nonempty finite set and $q\ge 1$ be an integer. We call $X$ a {\it $q$-fold complete residue set modulo $\A\Z^n$} if every residue class modulo $\A\Z^n$ contains exactly $q$ elements of $X$. In particular, every residue class is represented by $X$.
\end{definition}

\begin{proposition}\label{prop-basic} Let $X\subset\Z^n$ be a nonempty finite set and $q\ge 1$ be an integer. 
\begin{enumerate}
\item[(i)]
If $X$ is a $q$-fold complete residue set modulo $\A\Z^n$, then
for every $v\in\Z^n\setminus \A^T\Z^n$,
$$
P_X(e_1(v))=0.
$$

\item[(ii)]
If
$$
P_X(e_1(v))=0,
\qquad
v\in\Z^n\setminus \A^T\Z^n,
$$
then $X$ is a $\frac{\#X}{m}$-fold complete residue set modulo $\A\Z^n$
\end{enumerate}
\end{proposition}

\begin{proof}
For {\rm (i)}, since $X$ is a $q$-fold complete residue set modulo $\A\Z^n$, every element of $G_{\A}$ is represented exactly $q$ times by $X$. Therefore, for every nontrivial character $\chi_v$ of $ \Z^n_{\A} $,
$$0=q\sum_{[g]\in  \Z^n_{\A} }\chi_v([g])=\sum_{x\in X}\chi_v([x]).$$
Therefore, 
$$P_{X}(e_1(v))=0,$$
which proves {\rm (i)}.

For {\rm(ii)}, we define the representation-counting function $f_X$ on the finite group $ \Z^n_{\A} $ by
$$f_X([g])=\#\left\{x\in X: x\equiv g\pmod{\A\Z^n}\right\}.$$
For every nontrivial character $\chi_v$,
\begin{align*}
\widehat{f_X}(v)=\sum_{[g]\in  \Z^n_{\A} }f_X([g])\chi_v([g])=P_{X}(e_1(v))=0.
\end{align*} 
Therefore all nontrivial Fourier coefficients of $f_X$ vanish. By Fourier inversion on $ \Z^n_{\A} $, we have for every $[g]\in  \Z^n_{\A}$, 
$$f_X([g])=\frac1{m}\widehat{f_X}([0])=\frac{\#X}{m}.$$
The conclusion follows.
\end{proof}

\subsection{Kenyon's criterion} We next recall the Fourier criterion of Kenyon for self-affine tiles. 
For $k\geq1$, recall that 
$$
\mathcal Z_k(\A,\D )=
\left\{
v\in\Z^n:
P_\D \big(e_{k}(v)\big)=0
\right\}.
$$

\begin{theorem}[Kenyon's criterion {\cite{K92-2,LW96-1}}]
The self-affine set $T(\A,\D )$ is a tile if and only if for every
$\boldsymbol{0}\neq v\in\Z^n$, there exists an integer $k\geq1$ such that
$$
P_\D \big(e_k(v)\big)=0.
$$
Equivalently,
$$
\Z^n\setminus\{\boldsymbol{0}\}
\subset
\bigcup_{k=1}^{\infty}
\mathcal Z_k(\A,\D ).
$$
\end{theorem}

Since $\A$ is expanding, $\bigcap_{k=1}^{\infty}(\A^T)^k\Z^n=\{\boldsymbol{0}\}.$ Consequently, every $\boldsymbol{0}\neq v\in\Z^n$ can be written uniquely in the form $v=(\A^T)^\ell u$ for some $\ell\geq0$ and $u\in\Z^n\setminus \A^T\Z^n.$ Then,
$$
u\in {\mathcal Z}_k(\A,\D ) \quad \Longleftrightarrow\quad v\in {\mathcal Z}_{k+\ell}(\A,\D ).
$$
Thus Kenyon's criterion admits the following reduced form.

\begin{corollary}\label{coro-Kenyon}
    The self-affine set $T(\A,\D )$ is a tile if and only if 
    $$\Z^n\setminus \A^T\Z^n\subset\bigcup_{k=1}^{\infty}\mathcal Z_k(\A,\D ).$$
\end{corollary}

\subsection{The Noetherian property of the algebraic torus}\label{section:Noetherian}

We finally recall a few basic facts concerning the Zariski topology on the algebraic torus $(\mathbb C^\times)^n, \mathbb C^\times=\mathbb C\setminus\{0\},$
which will be used in the proof of Theorem~\ref{fzp-0}. We refer to \cite{CLS11,Har77} for standard facts on algebraic varieties and Laurent polynomials. Given a finite set  $E\subset\Z^n$, a Laurent polynomial on $(\mathbb C^\times)^n$ is defined as
$$
f(z)=\sum_{d\in E}c_d z^d, \quad c_d\in\mathbb C.
$$
The Laurent polynomials form the ring
$
\mathbb C[z_1^{\pm1},\ldots,z_n^{\pm1}].
$
For a family
$\mathcal F\subset\mathbb C[z_1^{\pm1},\ldots,z_n^{\pm1}]$, define
$$
Z(\mathcal F)
=
\left\{
z\in(\mathbb C^\times)^n:
f(z)=0\ \text{for every }f\in\mathcal F
\right\}.
$$
The sets of the form $Z(\mathcal F)$ are precisely the Zariski
closed subsets of $(\mathbb C^\times)^n$. The following standard fact will be essential in our argument.

\begin{proposition}\label{prop:torus-Noetherian}
The algebraic torus $(\mathbb C^\times)^n$, endowed with the
Zariski topology, is a Noetherian topological space. Consequently,
every descending chain
$$
V_1\supset V_2\supset V_3\supset\cdots
$$
of Zariski closed subsets of $(\mathbb C^\times)^n$ eventually
stabilizes.
\end{proposition}

We shall also need the monomial maps induced by integer matrices. For $\textbf{B}=(b_{ij})\in M_n(\Z)$, define
$$
\phi_\textbf{B}:(\mathbb C^\times)^n\longrightarrow(\mathbb C^\times)^n
$$
by
$$
\phi_\textbf{B}(z)
=
\left(
\prod_{j=1}^n z_j^{b_{1j}},
\ldots,
\prod_{j=1}^n z_j^{b_{nj}}
\right).
$$

\begin{lemma}\label{lem-conti}
For every $\textbf{B}\in M_n(\Z)$, the map $\phi_\textbf{B}$ is continuous
with respect to the Zariski topology.
\end{lemma}

\begin{proof}
For every $v\in\Z^n$,
$$
(\phi_\textbf{B}(z))^v=z^{\textbf{B}^Tv}.
$$
Hence, if $f$ is a Laurent polynomial, then $f\circ\phi_\textbf{B}$ is again
a Laurent polynomial. Therefore, for every family $\mathcal F$ of
Laurent polynomials,
$$
\phi_\textbf{B}^{-1}\bigl(Z(\mathcal F)\bigr)
=
Z\left(
\{f\circ\phi_\textbf{B}:f\in\mathcal F\}
\right),
$$
which is Zariski closed. Hence $\phi_\textbf{B}$ is continuous.
\end{proof}

In particular, let $\phi:=\phi_{\A^T}.$ For $j\ge 0$, let $\phi^j$ denote the $j-$fold iterate of $\phi$, with $\phi^{0}:=\operatorname{id} $. Then, for every $x\in\R^n$, we have $\phi(e(x))=e(\A^Tx).$
Therefore, for each $k\geq1$,
one has
$$
\phi^j(e_k(v))=e_{k-j}(v),
\qquad 0\leq j\leq k.
$$
In particular, if $v\in\Z^n$, then
$$
\phi^k(e_k(v))=e(v)=\mathbf 1,
$$
where $\mathbf 1=(1,\ldots,1)$.

\section{The finite zero-covering theorem}\label{sec:finite-zero-covering}
Throughout the rest of this paper, we write $T:=T(\A,\D ).$ In this section, we prove that every integral self-affine tile
$T$ satisfies the finite zero-covering condition in Theorem \ref{fzp-0}. The proof consists of three steps. We first decompose $T$ into finitely many measurable sets whose integer fibers share a common tiling complement. We then show that these fibers are uniformly distributed modulo $\A^k\Z^n$ after adding the finite-level digit sets $\D _{\A,k}$. Finally, a Noetherian argument on the algebraic torus yields a uniform bound for the depth of the zeros.

We begin with the following theorem of Lagarias and Wang.

\begin{theorem}[{\cite[Theorem 1.2]{LW96-2}}]\label{thm:integer-tiling-set}
If $T$ is an integral self-affine tile in $\R^n$, then it admits a tiling set $\J \subset\Z^n$.
\end{theorem}

Before introducing the integer fibers, we record a simple consequence 
of the self-affine identity. Recall that
$$
\D _{\A,k}
=
\D +\A\D +\cdots+\A^{k-1}\D .
$$
Since $T$ is a tile, it has positive Lebesgue measure and every element of $\D _{\A,k}$ has a unique
representation, and hence the sums in $\D_{\A, k}$ are direct \cite[Theorem 1.1]{LW96-2}.

For each $x\in [0, 1)^n$, define the integer fiber  
\begin{equation*}
 F_x:=\{z\in\Z^n:x+z\in T\}
\end{equation*}
Since $T$ is compact, the set $F_x$ is finite and uniformly bounded. Consequently, only finitely many different sets $F_x$ can occur. For a finite set $F\subset\Z^n$, the set
\begin{equation*}
 \begin{aligned}
T_F
&:=
\left\{
x\in[0,1)^n:F_x=F
\right\}
 \end{aligned}
\end{equation*}
is Borel measurable. Indeed, it follows directly from 
$$T_F=[0,1)^n\cap \bigcap_{k\in F}(T-k)\cap \bigcap_{a\in \underset{x\in[0, 1)^n}{\bigcup}F_x\setminus F}(\R^n\setminus(T-a)).$$

\begin{lemma}\label{lem-T}
 Suppose that $T$ tiles $\R^n$ with a tiling set
$ \J \subset\Z^n$. Then there exist finitely many finite
sets $F_1,\ldots,F_t\subset\Z^n$ and pairwise disjoint measurable sets
$T_1,\ldots,T_t\subset[0,1)^n$ of positive measure such that
\noindent

\medskip
{\rm(i)} $F_x=F_{\ell}, x\in T_{\ell}, 1\le \ell\le t$; 

\noindent

\medskip
{\rm(ii)} $[0, 1)^n=\bigsqcup_{\ell=1}^{t}T_{\ell}$ and $T=\bigsqcup_{\ell=1}^{t}\left(T_{\ell}+F_{\ell}\right)$ up to a null set;

\noindent

\medskip
{\rm(iii)}  $F_{\ell}\oplus\J=\Z^n, 1\le\ell\le t$
\end{lemma}
\begin{proof}
Since only finitely many different fiber types occur, discard those set $T_F$ of measure zero and denote the remaining ones by $T_1,\ldots, T_t$, with corresponding fibers $F_1,\ldots, F_t$. By construction,
$$[0, 1)^n=\bigsqcup_{\ell=1}^{t}T_{\ell}$$
up to a null set, and
$$T=\bigsqcup_{\ell=1}^{t}\left(T_{\ell}+F_{\ell}\right)$$
up to a null set. This proves {\rm(i)} and {\rm(ii)}.

It remains to prove the common tiling property. Since $T$ tiles $\R^n$ using $\J$, there exists a null set ${\mathcal N}\subset\R^n$ such that
\begin{equation}\label{eq-2.2.1}
\sum_{j\in \J }\one_{T}(y-j)=1,\quad \text{ for every } y\in\R^n\setminus {\mathcal N}.
\end{equation}
Set
$${\mathcal N}_0 := [0,1)^n\cap\left(\bigcup_{z\in\Z^n}({\mathcal N}-z)\right).$$
Since $\Z^n$ is countable, ${\mathcal N}_0$ is also a null set.
Fix $1\le \ell \le t$. Since $|T_\ell|>0$, we may choose $x\in T_\ell \setminus {\mathcal N}_0$, then  $x+z\notin{\mathcal N}$ for every $z\in\Z^n$. Therefore, by \eqref{eq-2.2.1}, for every $z\in\Z^n$, 
\begin{equation*}
\sum_{j\in \J }\one_{T}(x+z-j)=1.
\end{equation*}
Since $z-j\in\Z^n$ and $F_x=F_\ell$,
$$\one_{T}(x+z-j)=\one_{F_\ell}(z-j).$$
Consequently, 
$$ \sum_{j\in \J }\one_{F_\ell}(z-j)=1, \quad\ \forall\ z\in\Z^n.$$
Thus $F_\ell$ tiles $\Z^n$ using $\J$. Since $\ell$ was arbitrary, all the fibers $F_\ell$ share the common tiling set $\J$.
\end{proof}

We next relate fibers at different scales. For any $x\in[0, 1)^n$, define  
$$\tau:[0,1)^n\to[0,1)^n,\qquad \tau(x)=\A x\pmod{\Z^n}.$$
For $k\geq1$, set 
$$a_{k}(x)=\A^kx-\tau^k(x)\in\Z^n.$$ 
Here $\tau^{-k}(E)$ always denotes the fully $k$-fold preimage of $E$, not an inverse function.

\begin{lemma}\label{lem-tx}
    For every $x\in [0, 1)^n$ and $k\geq1$, 
    \begin{equation*}
     a_k(x)+\A^k F_x=(F_{\tau^k(x)}+\D_{\A,k})\cap (a_k(x)+\A^k\Z^n).   
    \end{equation*}
\end{lemma}

\begin{proof}Let  $a\in F_x$. Then $x+a\in T$, and hence 
$${\textbf A}^k(x+a)=\tau^k(x)+a_k(x)+{\textbf A}^ka\in{\textbf A}^kT.$$
It follows from the functional equation $ {\textbf A}^kT = T +\D_{\A,k}$ that there is some $t\in T$ and $d\in\D_{\A,k}$ such that
$$\tau^k(x)+a_k(x)+{\textbf A}^ka=t+d.$$
All terms other than $\tau^k(x)$ and $t$ differ by an integer vector, so $t=\tau^k(x)+\kappa$ for some $\kappa\in\Z^n$. Thus $\kappa\in F_{\tau^k(x)}$ and 
$$a_k(x)+{\textbf A}^ka=\kappa+d\in (F_{\tau^k(x)}+\D_{\A,k})\cap (a_k(x)+\A^k\Z^n).$$
This proves one inclusion.

Conversely, any element $v$ in  $(F_{\tau^k(x)}+\D_{\A,k})\cap (a_k(x)+\A^k\Z^n)$ has two expansions 
$$v=a_k(x)+{\textbf A}^ka=\kappa+d$$
with $a\in\Z^n$, $\kappa\in F_{\tau^k(x)}$, and $d\in\D_{\A,k}$. The definition of $F_{\tau^k(x)}$ implies that $\tau^k(x)+\kappa\in T$. Together with $d\in\D_{\A,k}$, this gives
$$\A^k(x+a)=\tau^k(x)+a_k(x)+{\textbf A}^ka=\tau^k(x)+\kappa+d\in \A^kT$$
 Since $\A^k$ is invertible, $x+a\in T$, which forces $a\in F_x$. Therefore, $v\in a_k(x)+{\textbf A}^kF_x.$ This proves the reverse inclusion.

\end{proof}

\begin{lemma}\label{lem:fiber-cardinality}
    There is an integer $r\geq1$ such that $$\#F_1=\cdots=\#F_t=r.$$
Moreover, $|T|=r$.
\end{lemma}

\begin{proof}
By the preceding fiber decomposition in Lemma \ref{lem-T} {\rm(iii)},
$$
    F_{\ell}\oplus \J=\Z^n,
    \qquad \ell=1,\ldots,t,
$$
Thus the tiling set $\J$ has natural density $1/\#F_{\ell}$. Since $\J$ is independent of $\ell$, all the sets $F_\ell$ have the same cardinality. Write
$$
\#F_1=\cdots=\#F_t=:r .
$$
Finally, by Tonelli's theorem,
$$
\begin{aligned}
    |T|
    &=\int_{[0,1)^n}\sum_{z\in\Z^n}
          \one_T(x+z)\,dx  \\
    &=\int_{[0,1)^n} \#F_x\,dx .
\end{aligned}
$$
Since the fiber types occurring on null sets may be discarded and $\#F_x=r$ for almost every $x\in [0,1)^n$, it follows that
$$
|T|=r.
$$
\end{proof}

The next result gives the main finite-level arithmetic structure of the integer fibers.

\begin{theorem}\label{th-com}
    For every $k\geq1$ and every $1\le\ell\le t$, 
    $$F_\ell+\D_{\A,k}=F_\ell\oplus\D_{\A,k},$$
    and this sum is a $r$-fold complete residue set modulo $\A^k\Z^n$.
\end{theorem}
\begin{proof}
    In Lemma \ref{lem-T} (ii), it has been proved  that the set 
    $$\widetilde{T}:=\bigsqcup_{\ell=1}^{t} T_{\ell}=[0, 1)^n, \quad \text{up to a null set.}$$ 
   Since $\tau^k$ is induced by the linear map $\A^k x$, it maps null sets
to null sets. Hence, for almost every $y\in T_{\ell}$, every point of
$\tau^{-k}(y)$ belongs to $\widetilde T$.

Fix such a point $y\in T_\ell$. For every
$x\in\tau^{-k}(y)$, the fiber $F_x$ is one of
$F_1,\ldots,F_t$, and therefore $\#F_x=r.$
The map $\tau^k$ has exactly
$|\det \A|^k=m^k$
preimages in $[0,1)^n$. Moreover,
$$
\left\{
a_k(x):x\in\tau^{-k}(y)
\right\}
$$
is a complete residue system modulo $\A^k\Z^n$. Indeed, suppose that $$a_k(x_1)\equiv a_k(x_2)\pmod{\A^k\Z^n},$$ 
Since $\tau^k(x_1)=\tau^k(x_2)=y,$ we have 
$$\A^k(x_1-x_2)=a_k(x_1)-a_k(x_2)\in \A^k\Z^n.$$
Hence $x_1-x_2\in\Z^n$ and since $x_1, x_2\in[0, 1)^n$, one has $x_1=x_2$. As there are exactly $m^k$ such points, the claim follows.

Applying Lemma \ref{lem-tx} to every inverse image $x$ of $y$, we have
\begin{equation*}
    a_k(x) + \A^kF_x = (F_\ell +\D_{\A,k}) \cap (a_k(x) + \A^k\Z^n).
\end{equation*}
Since $\#F_x=r$, every residue class modulo $\A^k\Z^n$ contains exactly $r$ elements of $F_\ell+\D_{\A,k}$. Hence
$$\#(F_\ell+\D_{\A,k})=rm^k.$$
On the other hand, $$\#F_\ell\cdot\#\D_{\A ,k}=r\cdot m^k.$$ 
Therefore,
$$\#F_\ell\cdot\#\D_{\A,k}=\#(F_\ell+\D_{\A,k}).$$
This proves the sum is direct, while the preceding residue count proves that this direct sum is a $r$-fold complete residue set modulo $\A ^k\Z^n$.
\end{proof}

The preceding theorem immediately gives a Fourier obstruction.

\begin{lemma}\label{lem:fourier-obstruction}
    Let $v\in\Z^n\setminus \A^T\Z^n$ and $k\geq1$. If $v\not\in\bigcup_{j=1}^k{\mathcal Z}_j(\A, \D)$, then 
    $$P_{F_\ell}(e_k(v))=0, \quad\quad  \ell=1, \cdots, t.$$
\end{lemma}

\begin{proof}
Since $v\notin \A^T\Z^n$, we also have $v\notin (\A^T)^k\Z^n$. By Theorem \ref{th-com}, $F_\ell \oplus \D_{\A,k}$ is a $r-$fold complete residue set modulo $\A^k\Z^n$. Using Proposition \ref{prop-basic} (i) with $\A$ replaced by $\A^k$,
$$P_{F_{\ell}}(e_k(v))\cdot P_{\D_{\A,k}}(e_k(v))=P_{F_{\ell}\oplus \D_{\A,k}}(e_k(v))=0, \quad \quad\ell=1, \cdots, t.$$
The product formula gives 
\begin{equation}\label{eq-P}
P_{\D_{\A,k}}(e_k(v))=\prod_{j=1}^{k}P_{\D}(e_j(v)).
\end{equation}
The assumption $v\not\in\bigcup_{j=1}^k{\mathcal Z}_j(\A, \D)$ means that every factor in \eqref{eq-P} is nonzero. Therefore,
$$P_{\D_{\A,k}}(e_k(v))\neq0.$$
It forces $P_{F_\ell}(e_k(v))=0$, $\ell=1, \cdots, t.$
\end{proof}

We can now establish the finite zero-covering condition for self-affine tiles, which, together with Corollary~\ref{coro-Kenyon}, completes the proof of Theorem~\ref{fzp-0}.

\begin{proof}[Proof of Theorem~\ref{fzp-0}]
Fix one positive-measure fiber, say $F=F_1$, and set
$$V:=\{z\in ({\mathbb C}^{\times})^n: P_{F}(z)=0 \}.$$
Since $P_{F}$ is a Laurent polynomial, $V$ is Zariski closed.
Let $\phi:=\phi_{\A^T}$ be the monomial map introduced in Section \ref{section:Noetherian}, that is 
$$\phi(z)
=
\left(
\prod_{j=1}^n z_j^{a_{1j}},
\ldots,
\prod_{j=1}^n z_j^{a_{nj}}
\right), \quad {\bf A}^T=(a_{ij}).$$
For $k\ge 1$, define
$$ V_k=\bigcap_{j=0}^{k-1}\phi^{-j}(V).$$
In Lemma \ref{lem-conti}, it has been proved that $\phi$ is Zariski continuous, so every $V_k$ is Zariski closed. Notice that $$V_{1}\supset V_{2}\supset V_3\supset\cdots.$$ 
Proposition \ref{prop:torus-Noetherian} gives that there is an
$N\ge 1$ such that $V_N=V_{N+1}$. Since 
$$V_{N+1}=V\cap \phi^{-1}(V_N),$$
we obtain 
$V_N\subset \phi^{-1}(V_N)$, which gives
\begin{equation}\label{eq-V}
    \phi(V_N)\subset V_N.
\end{equation}
Moreover, $P_{F}(\mathbf{1})=\#F>0$, so $\mathbf{1}\notin V$.

 Suppose, to the contrary, that there exists $v\in \Z^n\setminus \A^T\Z^n$ such that
 $$v\notin{\bigcup_{k=1}^{N}{\mathcal Z}_k(\A,\D)}.$$
 By Lemma~\ref{lem:fourier-obstruction}, $P_F(e_k(v))=0$ for all $1\le k\le N$. Therefore, $\phi^{j}(e_N(v))=e_{N-j}(v)\in V$ for all $0\le j\le N-1$ and then 
 $$e_N(v)\in \phi^{-j}(V),\qquad 0\le j\le N-1.$$
 Thus $e_N(v)\in V_N$. Since $V_N$ is forward invariant under $\phi$ (by \eqref{eq-V}),
 $$\mathbf{1}=\phi^N(e_N(v))\in V_N\subset V.$$
 However, $P_F(\mathbf{1})=\#F\neq0$. This contradiction proves the desired finite covering.
\end{proof}

\section{Skew-product structures and periodic tilings}\label{sec:skew-product}

In this section, we will prove our two main theorems.

\begin{proof}[Proof of Theorem~\ref{main2}]
The sufficiency was established in \cite[Theorem~1.1]{LR25}. We prove the necessity. Suppose that $T$ is a self-affine tile. According to \cite[Theorem 1.1]{LW96-2},  every element of each iteration $\D_{\A, k}$ has a unique representation. It means the sums in $\D_{\A, k}$ are all direct for all $k\geq1$.

Let $N\geq1$ be an integer such that
\begin{equation}\label{finite-cover}
\Z^n\setminus \A^T\Z^n\subset \bigcup_{k=1}^{N}\mathcal{Z}_{k}(\A,\D).
\end{equation}
and let  $$ \W=\{w_0,w_1,\ldots,w_{s-1}\}\subset\D_{\A, N} $$ be a set  of representatives of all residue classes modulo $\A^N\Z^n$ which occur in
$\D_{\A, N}$. We further obtain the decomposition along $\A^N$:
\begin{equation}\label{eq-decom-D}
\D_{\A, N}=\bigsqcup_{j=0}^{s-1}\left(\{w_j\}+\A^N\E_j\right),
\end{equation}
where $\E_j=\{x\in \Z^n: w_j+\A^Nx\in\D_{\A, N}\}$ and uniqueness at level $2N$ gives
\begin{equation}\label{eq-2.1}
\D_{\A,2N}=\D _{\A,N} \oplus \A^N\D _{\A,N}=\bigsqcup_{j=0}^{s-1}\left(\{w_j\}+\A^N(\E_j\oplus\D_{\A, N})\right).
\end{equation}
In particular, as  $\W\subset\D_{\A, N}$, the sums in every set $\E_j+\W$ are direct.

\medskip

{\bf Step 1: For every $j$,  $\E_j\oplus \W$ contains a complete residue  system modulo $\A^N\Z^n$.}

\medskip
Indeed, since $w_0,\ldots,w_{s-1}$ are distinct modulo  $\A^N\Z^n$,
the corresponding rows of the Fourier matrix of $ \Z^n_{\A^N} $ are
linearly independent. Hence there exist
$\{a_0, a_1, \cdots, a_{s-1}\}\subset\Z^n$ 
such that the $s\times s$ matrix
$$H=\Big(e( w_j\cdot (\A^{T})^{-N}a_{\ell})\Big)_{0\le \ell, j\le s-1}$$
is invertible. Fix $v\in\Z^n\setminus (\A^{T})^N\Z^n$, and set 
$$z_{\ell}:=v+(\A^T)^Na_\ell\in\Z^n\setminus (\A^{T})^N\Z^n,$$ 
then 
$$(\A^T)^{-j_{\ell}}z_{\ell}\in \Z^n\setminus \A^{T}\Z^n$$
for some $0\le j_{\ell}\le N-1$. So, from \eqref{finite-cover}, there is a $1\le k_{\ell}\le N$ such that 
  $$P_{\D}\left(e_{j_{\ell}+k_{\ell}}(z_{\ell})\right)=0.$$
Since $1\le k_{\ell}+j_{\ell}\le 2N-1$, the product formula gives
$$P_{\D_{\A, 2N}}\left(e_{2N}(z_{\ell})\right)=\prod_{k=1}^{2N} P_\D  \left(e_{k}( z_\ell )\right) = 0.$$
We now expand the above equation by using the  decomposition  \eqref{eq-2.1} that, for all $0\le \ell\le s-1$, 
 \begin{align*}
   0 &=P_{\D_{\A, 2N}}\left(e_{2N}(z_{\ell})\right)\\
    &=\sum_{j=0}^{s-1}\sum_{d\in\E_j\oplus\D_{\A, N} }e_{2N}(z_{\ell})^{w_j}e_N(z_{\ell})^d\\
    &=\sum_{j=0}^{s-1}e_{N}(a_{\ell})^{w_j}\cdot e_{2N}(v)^{w_j}P_{\E_j\oplus\D_{\A, N} }(e_N(v)).
   \end{align*}
The invertibility of $H$ forces each
$$e_{2N}(v)^{w_j}P_{\E_j\oplus\D_{\A, N} }(e_N(v))=0$$
and so $$P_{\E_j\oplus \D_{\A, N}}(e_N(v))=0,\quad   \forall\ v\in\Z^{n}\setminus (\A^T)^{N}\Z^n.$$
Applying Proposition \ref{prop-basic} {\rm(ii)} with $\A$ replaced by  $\A^{N}$,  we have proved that each $\E_j\oplus\D_{\A,N}$ represents every residue class modulo $\A^N\Z^n$ and so does $\E_j\oplus\W$ since $\W\equiv \D_{\A,N}\pmod{\A^N\Z^n}$.

\medskip

{\bf Step 2: $\D_{\A,N}$ is a skew-product-form digit set with respect to $\A^N$.}

\medskip

Since $\E_j\oplus\W$ represents every residue class modulo $\A^N\Z^n$, we have 
$$s\cdot\#\E_j\ge|\det \A|^N=m^N.$$ On the other hand, from \eqref{eq-decom-D}, 
$$m^N=\#\D_{\A,N}=\sum_{j=0}^{s-1}\#\E_j.$$
Consequently equality must hold in every preceding inequality: 
$$s\cdot\#\E_j=m^N,\quad 0\le j\le s-1.$$
Thus each covering  $\E_j\oplus\W$ of the finite group $ \Z^n_{\A^N} $ has exactly $m^N$ representing pairs, and
hence is unique. Therefore, $\E_j\oplus \W$ is a complete residue set modulo $\A^N\Z^n$ for every $j$. This proves that the decomposition \eqref{eq-decom-D} of $\D_{\A,N}$ is of skew-product form with respect to $\A^N$.
\end{proof}

Under the assumption that the associated self-affine set is a tile, Li and Rao \cite{LR25} proved the equivalence between a skew-product-form digit set and the periodicity of a self-replicating set. Here we show that the tiling assumption can be omitted.

\begin{theorem}\label{thm:self-replication}
Let $\A\in M_n(\Z)$ be expanding, let $\D\subset \Z^n$ satisfy $\#\D=|\det \A|:=m$, and let $N\ge 1$. Then the following statements are equivalent:
\begin{enumerate}
\item[(i)] $\D_{\A,N}$ is a skew-product-form digit set with respect to $\A^N$;
\item[(ii)] there exists a nonempty set $\J_{R}\subset\Z^n$ satisfying $\J_R=\A\J_R+\D$ and having period lattice $\A^N\Z^n$.
\end{enumerate}
\end{theorem}

\begin{proof}
{\bf (i) $\Rightarrow$ (ii).} Since $\D_{\A,N}$ is a skew-product-form digit set with respect to $\A^N$,
 there exist $\emptyset \neq \mathcal W=\{w_j\}_{j=0}^{s-1}\subset\Z^n $ and finite sets $\mathcal E_0,\ldots,\mathcal E_{s-1}\subset\Z^n$
such that
\begin{equation*}
 \D_{\A,N}=\bigsqcup_{j=0}^{s-1}\left(
\{w_j\}+\A^N\mathcal E_j
\right),
\end{equation*}
and
$\mathcal W\oplus\mathcal E_j$ is a complete residue set modulo $\A^N\Z^n$ for every $j=0,\ldots,s-1$.
Let
\begin{equation}\label{eq-1.2}
  \J_R=\W+\A^N\Z^n=\D_{\A, N}+\A^N\Z^n,  
\end{equation}
which is nonempty and has period lattice $\A^N\Z^n$. Since $\W\oplus \E_j$ is a complete residue set modulo $\A^N\Z^n$, we have
\begin{equation}\label{eq-1.1}
\mathcal E_j+\D _{\A,N}+\A^N\Z^n
=
\Z^n,
\qquad 0\leq j\leq s-1.
\end{equation}
Substituting \eqref{eq-1.1} into \eqref{eq-1.2}, one can conclude that
\begin{align*}
\J_R
&=\W+\A^N\Z^n\\
&=\bigcup_{j=0}^{s-1}\left(\{w_j\}+\A^N\left(\E_j+\D_{\A, N}+\A^N\Z^n\right)\right)
    &&\text{by \eqref{eq-1.1}}\\
&=\D_{\A, 2N}+\A^{2N}\Z^n\\
&=\D+\A(\D_{\A, 2N-1}+\A^{2N-1}\Z^n)\\
&\subset\D+\A(\D_{\A, N}+\A^N\Z^n)
    &&\text{by \eqref{eq-1.2}}\\
&=\D+\A\J_R.
\end{align*}
Conversely,
\begin{align*}
\D+\A\J_R
&=\D+\A\D_{\A, N}+\A^{N+1}\Z^n\\
&=\D_{\A, N}+\A^N(\D+\A\Z^n)\\
&\subset\D_{\A, N}+\A^N\Z^n\\
&=\J_R. &&\text{by \eqref{eq-1.2}}
\end{align*}
We have proved that  $\J_R=\A\J_R+\D$ is a self-replicating set.

{\bf (ii) $\Rightarrow$ (i).}
Iterating the self-replication identity gives 
\begin{equation}\label{eq:J-iterated}
 \J _R=\D _{\A,N}+\A^N \J _R.
\end{equation}
Because $\A^N\Z^n$ is a period lattice of the nonempty set $\J_R$, choose distinct representatives
$$\emptyset \neq \mathcal W=\{w_0, w_1, \cdots, w_{s-1}\}\subset\Z^n$$ 
such that 
\begin{equation}\label{eq:J-period}
 \J _R=\mathcal W+\A^N\Z^n.
\end{equation}
We first show that $\#\D_{\A, N}=m^N$. Reduce the sets in \eqref{eq:J-period} modulo $\A^{2N}\Z^n$. Since each coset $w_j+\A^{N}\Z^n$ splits into exactly $m^N$ residue
classes modulo $\A^{2N}\Z^n$,
$$\#(\J_R\pmod{\A^{2N}\Z^n})=s\cdot m^N.$$
Moreover, $\A^N\J_R=\A^N\W+\A^{2N}\Z^n,$  and the elements of $\A^N\W$ are distinct modulo $\A^{2N}\Z^n$, hence
$$\#(\A^N\J_R\pmod{\A^{2N}\Z^n})=s.$$
The right-hand side of \eqref{eq:J-iterated}  is a union of at most $\#\D_{\A,N}$ translates of $\A^N\J_R$. Therefore,  $$s\cdot \#\D_{\A, N}\geq s\cdot m^N.$$ Note that $s>0$ and the reverse inequality 
$$\#\D_{\A, N}\le (\#\D)^N=m^N$$
is immediate.  So, every expansion in $\D_{\A, N}$ is unique.

Reducing \eqref{eq:J-iterated} modulo $\A^N\Z^n$, we obtain
$$
\mathcal W\equiv\D _{\A,N}\pmod{\A^N\Z^n}.$$
Hence $\D _{\A,N}$ admits a disjoint decomposition along $\A^N$:
\begin{equation}\label{eq:DN-fiber}
\D _{\A,N}
=
\bigsqcup_{j=0}^{s-1}
\bigl(
\left\{w_j\right\}
+
\A^N\mathcal E_j
\bigr),
\end{equation}
where $\mathcal E_j:=\left\{x\in\Z^n:w_j+\A^Nx\in\D _{\A,N}\right\}$. 
Substituting \eqref{eq:DN-fiber} and \eqref{eq:J-period} into \eqref{eq:J-iterated}, and comparing the fibers modulo $\A^N\Z^n$, yields
$$
\mathcal E_j+\mathcal W+\A^N\Z^n
=
\Z^n
\qquad
0\leq j\leq s-1.
$$
which means that $\E_j+\W$ covers every residue class modulo $\A^N\Z^n$. Therefore, $$s\cdot\#\E_j\ge m^N.$$ On the other hand,
$$m^N=\#\D_{\A,N}=\sum_{j=0}^{s-1}\#\E_j.$$
Consequently equality must hold in every preceding inequality: 
$$\#\E_j=\frac{m^N}{s},\quad 0\le j\le s-1.$$
Thus each covering $\E_j+\W$ of the finite group $ \Z^n_{\A^N} $ has exactly $m^N$ representing pairs and is therefore unique. Thus, $\E_j\oplus\W$ is a complete residue set modulo $\A^N\Z^n$ for every $j$. This proves that $\D_{\A,N}$ is a skew-product-form digit set with respect to $\A^N$.
\end{proof}

\begin{proof}[Proof of Theorem~\ref{multi-tiling}]
Suppose that $T$ is a self-affine tile. By Theorem~\ref{main2}, there is an integer $N\ge 1$ such that $\D_{\A,N}$ is of skew-product form with respect to $\A^N$, that is,
\begin{equation*}
 \D_{\A,N}=\bigsqcup_{j=0}^{s-1}\left(
\{w_j\}+\A^N\mathcal E_j
\right),
\end{equation*}
and
$\mathcal W\oplus\mathcal E_j$ is a complete residue set modulo $\A^N\Z^n$ for every $j=0,\ldots,s-1$, where $\W=\{w_j\}_{j=0}^{s-1}$.
Theorem~\ref{thm:self-replication} shows that there exists a nonempty set $\J_R\subset\Z^n$ satisfying $\J_R=\A\J_R+\D$ and $\J_R=\W\oplus \A^N\Z^n$, which is fully periodic and self-replicating.

Let $r=|T|$. By Lemma \ref{lem:fiber-cardinality},
$\#F_\ell=r$ for every positive-measure fiber $F_\ell$.
Moreover, Theorem~\ref{th-com} shows that
$F_\ell\oplus \D_{\A,N}$ is an $r$-fold complete residue set modulo $\A^N\Z^n$. Notice that all $\E_j$ have the same cardinality, and denote it as $h$. 
Let $f_{F_\ell}$ be the representation-counting function of $F_\ell$ on $ \Z^n_{\A^N} $. Since the representation-counting function  of $\D_{\A,N}$ is $h\one_{\W}$, we have
$$
    h(f_{F_\ell} \ast \one_{\W})=r\one_{ \Z^n_{\A^N} }.
$$
Hence $h\mid r$, and, writing
$$
q:=\frac rh,
$$
we obtain
$$
f_{F_\ell}\ast \one_{\W}=q \one_{ \Z^n_{\A^N} }.
$$
$\J_R=\W\oplus \A^N\Z^n$ gives that
$$
\sum_{j\in\J_R}\one_{F_\ell}(z-j)=q, \qquad z\in\Z^n.
$$
For almost every $x\in[0,1)^n$, the fiber $F_x$ is one of the
positive-measure fiber types. Therefore, for every $z\in\Z^n$,
$$\sum_{j\in\J_R}\one_T(x+z-j)=\sum_{j\in\J_R}\one_{F_x}(z-j)=q.$$
Since $z$ is arbitrary, by taking $y=x+z$, we have
$$
\sum_{j\in \J _R}\one_T(y-j)=q,\quad\text{for almost every $y\in\R^n$.}
$$
Hence $\J_R$ is a $q$-fold tiling set of $T$, with $q\mid |T|$.
\end{proof}

\section{The PTC for planar self-affine tiles}\label{sec:planar}
In this section, we focus on the integral self-affine tile $T$ in $\R^2$. By Theorem~\ref{thm:integer-tiling-set}, it admits a tiling set $\J\subset\Z^2$. In order to replace the common tiling set $\J$ by a common fully periodic one, we use the following periodicity theorem of Greenfeld and Tao for indicator-function solutions of convolution equations on $\Z^2$.

\begin{theorem}[{\cite[Theorem~1.9]{GT25}}]\label{thm:GT-periodicity}
Let $f$ be a finitely supported function from $\Z^2$ to $\Z$, and let $c\geq1$ be an integer. If the equation $f\ast a=c$ has an indicator-function solution $a=\one_{\J}$ for some $\J\subset\Z^2$, then it also has an indicator-function solution $a_P=\one_{\J_P}$ for some fully periodic set $\J_P\subset\Z^2$.
\end{theorem}

We now simultaneously encode all the tiling equations into one convolution
equation.

\begin{lemma}\label{lem-pT}
Let $ F_1,\ldots,F_t\subset\Z^2$ be the finite sets obtained in Lemma \ref{lem-T}.  Then they have a common  fully periodic tiling set $\J _P\subset\Z^2$.
\end{lemma}

\begin{proof}
Choose an integer $\gamma>\max_{1\leq\ell\leq t}\#F_\ell$ and define
$$f:=\sum_{\ell=1}^{t}\gamma^{\ell-1} \one_{F_\ell}.$$
Then $f:\Z^2\to\Z$ is finitely supported. By Lemma \ref{lem-T} (i), the sets $F_\ell$ have a common tiling  set $\J$, equivalently 
$$
   \one_{F_\ell}\ast\one_{\J}=1, \quad 1\le \ell\le t.
$$
Hence, by linearity of convolution, we have
$$
\begin{aligned}
    f\ast \one_{\J} = \sum_{\ell=1}^{t} \gamma^{\ell-1} (\one_{F_\ell}\ast\one_{\J})= \sum_{\ell=1}^{t}\gamma^{\ell-1}:=c.
\end{aligned}
$$
Thus the convolution equation $f\ast a=c$ has the indicator-function solution $ a=\one_{\J}$. 
By Theorem~\ref{thm:GT-periodicity}, there exists a fully periodic
set $\J_P\subset\Z^2$ such that
\begin{equation}\label{eq:periodic-encoded}
    f\ast \one_{\J_P}=\sum_{\ell=1}^{t}\gamma^{\ell-1} (\one_{F_\ell}\ast\one_{\J_P})=c.
\end{equation}

It remains to show that the single encoded identity
\eqref{eq:periodic-encoded} implies all the individual identities. Fix $z\in\Z^2$, and for $1\leq \ell\leq t$, we define
$$
n_\ell(z):=\sum_{j\in\J_P}\one_{F_\ell}(z-j).
$$
Clearly, one has 
$$ 0\leq n_\ell(z)\leq \#F_\ell<\gamma.$$
On the other hand, \eqref{eq:periodic-encoded} gives
\begin{equation}\label{eq:base-B}
    \sum_{\ell=1}^{t}\gamma^{\ell-1}n_\ell(z) = \sum_{\ell=1}^{t}\gamma^{\ell-1}.
\end{equation}
Reducing \eqref{eq:base-B} modulo $\gamma$, we obtain
$$
    n_1(z)\equiv 1\pmod \gamma.
$$
Since $0\leq n_1(z)<\gamma$, it follows that $n_1(z)=1$.
Subtracting $1$ from both sides of \eqref{eq:base-B} and dividing by $\gamma$, we
obtain
$$
    n_2(z)+\gamma n_3(z)+\cdots+\gamma^{t-2}n_t(z)
    =
    1+\gamma+\cdots+\gamma^{t-2}.
$$
Reducing again modulo $\gamma$ gives $n_2(z)=1$.
Repeating this argument inductively yields
$$
    n_\ell(z)=\sum_{j\in\J_P}\one_{F_\ell}(z-j)=1,
    \qquad 1\leq \ell\leq t.
$$
Since $z\in\Z^2$ was arbitrary, we have proved that $F_\ell$ tiles $\Z^2$ using the common fully periodic tiling set $\J_P$. 
\end{proof}

Now we are ready to prove Theorem~\ref{thm:planar-ptc}.

\begin{proof}[Proof of Theorem~\ref{thm:planar-ptc}]
Suppose that $T$ is an integral self-affine tile in $\R^2$. By Theorem~\ref{thm:integer-tiling-set}, it admits a tiling set $\J\subset\Z^2$. Applying Lemma~\ref{lem-T}, there exist finite sets $F_1, \ldots, F_t\subset\Z^2$ admitting the same tiling set $\J$ and pairwise disjoint measurable sets $T_1, \ldots, T_t\subset[0,1)^2$ such that
\begin{equation*}
[0,1)^2=\bigsqcup_{\ell=1}^{t}T_\ell,\quad \text{up to a null set},
\end{equation*}
and
\begin{equation*}
T=\bigsqcup_{\ell=1}^{t}\left(T_{\ell}+F_\ell\right),\quad \text{up to a null set}.
\end{equation*}
By Lemma \ref{lem-pT}, there exists a fully periodic set $\J_P\subset\Z^2$ such that
\begin{equation*}
\sum_{j\in\J_P}\one_{F_\ell}(z-j)=1,\quad
z\in\Z^2,\quad 1\leq\ell\leq t.
\end{equation*}
For almost every $x\in[0,1)^2$, the fiber $F_x$ is one of the
positive-measure fiber types. Therefore, for every $z\in\Z^2$,
$$\sum_{j\in\J_P}\one_T(x+z-j)=\sum_{j\in\J_P}\one_{F_x}(z-j)=1.$$
Thus
$$
\sum_{j\in\J_P}\one_T(y-j)=1
$$
for almost every $y\in\R^2$. Hence $T$ tiles $\R^2$ at level one with the fully periodic tiling set $\J_P$.
\end{proof}

\section{Remarks and examples}\label{sec:examples}

In this section, we give several remarks and examples to illustrate the
main results obtained in this paper.

\subsection{ One-dimensional self-similar tiles.}

In this subsection, we assume that $b>1, {\bf A}=[b]$ and $\#\D=b$. The existence of $N$ was given by Li and Rao \cite{LR25}. Here we give its explicit description.  
Note that 
 $$T(b,{\D}+d)=T(b, \D)+\sum_{j=1}^{\infty}b^{-j}d.$$
The self-similar tiling property of $T (b, \D)$ is invariant under such translations. Hence,  without loss of generality, we assume that $0\in\D\subset\Z^+$. In these settings, 
$${\mathcal Z}_k(b, \D)=\{v\in\Z: P_\D (e(v/b^k))=0\}.$$
We use $\Phi_s(x)$ to denote the $s-$th cyclotomic polynomial, i.e., the minimal polynomial of the primitive $s$-th
root of unity $e(1/s)$.  The degree of $\Phi_s(x)$ is equal to the Euler-phi function $\varphi(s)$, the number of relatively prime integers in $\{1, 2, \cdots, s-1\}$.  We denote 
$$ d_{\max} =\max \{d: d\in\D \},\quad N_0=1+\left\lfloor \log_2 d_{\max} \right\rfloor.$$

\begin{theorem}\label{cyc-poly}
    Let $b>1$ and $0\in\D\subset\Z^+$ satisfy $\#\D=b$. If $T(b, \D)$ is a tile, then $$\Z\setminus b\Z\subset \bigcup_{k=1}^{N_0}{\mathcal Z}_k(b, \D).$$
\end{theorem}
\begin{proof}
   For any $v\in \Z\setminus b\Z$, according to the Kenyon's criterion, there is an integer $k\geq1$ such that $P_{\D}(e(v/b^k))=0$. We factorize $b$ and $v$ into primes as 
$$b=p_1^{\beta_1}p_2^{\beta_2}\cdots p_m^{\beta_m},\quad v=p_1^{\alpha_1}p_2^{\alpha_2}\cdots p_m^{\alpha_m}v',$$
where $\beta_j\geq1, \alpha_j\geq0$ for every $j=1, \cdots, m.$ Let 
$$s=\prod_{\alpha_j<k\beta_j }p_j^{\beta_jk-\alpha_j}.$$
Then $(e(v/b^k))$ is a primitive $s$-th
root of unity and so $\Phi_s(x)|P_{\D}(x).$ This gives the degree of $\Phi_s(x)$ must satisfies
\begin{equation}\label{eq6.1}
\varphi(s)\le \deg P_{\D}(x)=d_{\max}.
\end{equation}
As $v\not\in b\Z$, there is an index $j_v$ such that $\alpha_{j_v}<\beta_{j_v}$
and so $p_{j_v}^{\beta_{j_v}k-\alpha_{j_v}}$ is a factor of $s$. Note that
$$\varphi\left(p_{j_v}\right)=p_{j_v}^{\beta_{j_v}k-\alpha_{j_v}-1}(p_{j_v}-1).$$
It together with \eqref{eq6.1} implies that
$$p_{j_v}^{\beta_{j_v}(k-1)}(p_{j_v}-1)\le p_{j_v}^{\beta_{j_v}k-\alpha_{j_v}-1}(p_{j_v}-1)\le d_{\max}.$$
Since both $\log_x a$ and $\frac1{x}$ are decreasing functions in $(1, \infty)$, the above inequality forces  
$$k\le 1+ \log_{p_{j_v}^{\beta_{j_v}}} \frac{d_{\max}}{p_{j_v}-1}\le 1+\log_2 d_{\max}$$
One can conclude that $k\le N_0$. Therefore, $v\in \bigcup_{k=1}^{N_0}{\mathcal Z}_k(b, \D)$ and we have complete the proof.
\end{proof}

For every $N\geq N_0$, the finite zero-covering condition is still satisfied. Combining this with Theorem \ref{main2} immediately yields the following result.

\begin{corollary} Let $b$ be an integer and $0\in\D\subset\Z^+$ satisfy $\#\D=b$. If $T(b, \D)$ is a tile, then $\D_{b, N}$ is of skew product-form for all $N\geq N_0$.
\end{corollary}

\subsection{The conjecture of Li and Rao}

Let
$$
 \J _0^{\ast}=\Z^n,
\qquad
 \J _k^{\ast}=\A \J _{k-1}^*+\D ,
\quad k\geq1.
$$
Clearly,
$$
 \J _0^\ast\supset
 \J _1^\ast\supset
 \J _2^\ast\supset\cdots.
$$
Li and Rao \cite{LR25} conjectured that if $T$ is an integral
self-affine tile, then this decreasing sequence stabilizes after
finitely many steps; namely, there exists $N\geq1$ such that
$$
 \J _{N+1}^\ast= \J _N^\ast.
$$
They proved the conjecture in the one-dimensional case.  Theorem \ref{main2} gives an affirmative answer to this conjecture in arbitrary
dimensions.

\begin{corollary}
   If $T$ is an integral self-affine tile, then there is an
$N\geq1$ such that $ \J _{N+1}^\ast= \J _N^\ast.$
\end{corollary}

\begin{proof}
By Theorem \ref{main2}, there exists $N\geq1$ such that
$\D _{\A,N}$ is a skew-product-form digit set with respect to
$\A^N$. By Theorem~\ref{thm:self-replication},
$$
 \J_R 
=
\D _{\A,N}+\A^N\Z^n
$$
satisfies
$$
 \J_R =\A \J_R +\D .
$$
Since $ \J_R = \J _N^*$, we obtain
$$
 \J _{N+1}^*
=
\A \J _N^*+\D 
=
 \J _N^*.
$$
\end{proof}

In this sense, the infinite iteration associated with a self-affine tile eventually becomes
stationary at a finite arithmetic level.

\subsection{Relation with the planar result of An and Lau}

Denote by $\Z[\A ,\D ]$ the smallest $\A $-invariant lattice
containing $\D -\D $.
We say that the affine pair $(\A ,\D )$ is primitive if $\Z[\A ,\D ]=\Z^n. $ An--Lau's \cite{AL19} planar prime-dilation characterization is formulated after a
reduction to the primitive case, whereas Theorem~\ref{main2} does
not require such a reduction.  Even for a non-primitive affine pair,
the tile property is detected intrinsically by the skew-product
structure of a finite iterate $\D _{\A ,N}$.

\begin{example}
Let $\A=2I_2$ and
$\D  =\{0,3\}\times\{0,1\}.$
Then $\#\D =4=|\det \A|$. The associated self-affine set has the product form
$$
T
=
T(2,\{0,3\})\times T(2,\{0,1\})
=
[0,3]\times[0,1].
$$
Hence $T$ is a self-affine tile, using a lattice $3\Z\times\Z$.  
On the other hand, $(\A,\D )$ is not primitive since $\Z[\A,\D ]
=
3\Z\times\Z.
$
Nevertheless, $\D $ itself is a complete residue system modulo
$2\Z^2$. Consequently, it is already a skew-product-form
digit set with respect to $\A$, and one may take $N=1$ in
Theorem \ref{main2}.
\end{example}

\end{document}